\documentclass[12pt]{amsart}
\usepackage[utf8]{inputenc}
\usepackage[T1]{fontenc}
\usepackage{amsmath, amssymb, amsthm}
\usepackage{geometry}
\newtheorem{theorem}{Theorem}
\newtheorem{proposition}{Proposition}
\newtheorem{lemma}{Lemma}
\newtheorem{corollary}{Corollary}

\newtheorem{definition}{Definition}
\newtheorem{example}{Example}
\newcommand{\U}{\mathcal{U}}
\newcommand{\Lc}{\mathcal{L}}
\newcommand{\R}{\mathbb{R}}

\title[Convexity of the embedding parameter sets]{Convexity of the embedding parameter sets of some analytic function spaces}
\author {Beno\^it F. Sehba}
\address{Department of Mathematics, University of Ghana, PO. Box LG 62 Legon, Accra, Ghana }
\email{ bfsehba@ug.edu.gh}

\begin{document}

\maketitle

\begin{abstract}
In this note, we study the geometric structure of the parameter sets governing continuous embeddings between weighted Bergman-Orlicz spaces. First, for a fixed pair of growth functions, we show that the set of admissible weight exponents $(\alpha, \beta)$ is convex, provided the growth functions satisfy specific log-convexity and log-concavity conditions of the inverses. Second, we consider the dual problem where the weight exponents are fixed. We prove that the collection of growth function pairs that yield such an embedding is log-convex under a natural interpolation of their inverses. We then obtain interpolated embeddings between Bergman-Orlicz spaces. 
\end{abstract}



\section{Introduction}

In recent years, the study of embeddings between weighted Bergman spaces has expanded to the Orlicz setting, where the growth of analytic functions is measured by a Young function rather than a pure power. A central problem systematically investigated by Dje and Sehba in \cite{DieSehba2021,DieSehba2023,Sehba} is to find sharp relations between the growth functions $\Phi_1, \Phi_2$ and the weight exponents $\alpha, \beta$ such that the Carleson embedding
\[
A^{\Phi_1}_{\alpha} \hookrightarrow A^{\Phi_2}_{\beta}
\]
holds. Here $A^{\Phi}_{\alpha}$ denotes the Bergman-Orlicz space on the unit ball of $\mathbb{C}^n$ or the upper half-plane. Their work reveals that the embedding condition is equivalent (under appropriate conditions on the growth functions) to the following inequality 
\[
\Phi_1^{-1}\!\left( t^{2+\alpha} \right) \le C \, \Phi_2^{-1}\!\left( t^{2+\beta} \right), \qquad \forall t\ge 1,
\]
for some constant $C>0$ (for classical Bergman spaces, see \cite{HedenmalmKorenblumZhu,Luecking,Oleinik,ZhuBook,ZhuBergman} and the references therein). This inequality, which compares the inverted growth functions at different exponents, is the key to understanding the geometry of admissible parameters.

The present paper takes this condition as its starting point and studies two natural convexity properties that arise from it. 

More precisely, we show that for fixed growth functions $\Phi_1$ and $ \Phi_2$ whose inverses are log convex and log concave respectively, the set
\[
\mathcal{E} = \left\{ (\alpha,\beta) \in \R^2 \;\middle|\;
\alpha,\beta \ge -1,\quad
\exists\, C,K > 0 \text{ s.t. }
\Phi_1^{-1}\!\left( t^{2+\alpha} \right) \le C\,
\Phi_2^{-1}\!\left( K t^{2+\beta} \right) \;\forall\, t>0
\right\}
\]
is convex. We then also obtain that $\mathcal{E}$ is an epigraph of a related function.

For the dual problem, we show in two different ways that 
for fixed $(\alpha,\beta)$, the set
\[
\mathcal{F} = \left\{ (\Phi_1,\Phi_2) \in \mathcal{G}^2 \;\middle|\;
\exists\, C > 0 \text{ s.t. }
\Phi_1^{-1}\!\left( t^{2+\alpha} \right) \le C\,
\Phi_2^{-1}\!\left( t^{2+\beta} \right) \;\forall\, t>0
\right\}
\]
is log-convex under log-convex interpolation of the growth functions. Here $\mathcal{G}$ is the set of growth functions on $[0,\infty)$. This last result along with some invariance of the monotonicity of ratios of growth functions allows us to obtain interpolation of embeddings between Bergman-Orlicz spaces without going through the machinery of complex interpolation theory (see for example \cite{ZhuZhao}). This on its own, is a strong motivation for this paper.

\medskip
These convexity properties have natural operational and conceptual consequences. The convexity of $\mathcal{E}$ implies that the admissible exponent region is an epigraph $\{(\alpha,\beta): \beta \ge \beta^*(\alpha)\}$ with $\beta^*$ convex and nondecreasing. This yields a phase diagram in the usual way; points above the curve give embeddings while points below do not. The convex boundary guarantees that the optimal (smallest) $\beta$ for a given $\alpha$ is a convex function, so any local minimum is global, enabling convex optimization for sharp embedding exponents. Moreover, if embeddings hold for two exponent pairs, they hold for all convex combinations, providing stability under perturbations and a powerful extrapolation tool. For more on convex analysis, we refer the interested reader  to \cite{Rockafellar1970}.

For $\mathcal{F}$, log-convexity means the admissible set is closed under geometric averaging. The optimal embedding constant $C_{\min}(\Phi,\Psi)$ satisfies
\[
C_{\min}(\Phi_\theta,\Psi_\theta) \le C_{\min}(\Phi_0,\Psi_0)^{1-\theta} C_{\min}(\Phi_1,\Psi_1)^{\theta},
\]
where $\Phi_\theta$ and $\Psi_\theta$ are interpolates of $(\Phi_0,\Phi_1)$ and $(\Psi_0,\Psi_1)$ respectively. Thus, $\log C_{\min}$ is convex along interpolations. This gives explicit control over the embedding constant for interpolated spaces.  Together, the two results reveal a duality between exponent space and function space, suggesting a deeper variational principle underlying the embedding criteria for Bergman-Orlicz spaces.

\medskip
The paper is organized as follows. Section 2 recalls background on growth functions, log-convexity, log-concavity, and interpolation. In Section 3, we present and prove our main results. We end the paper with Section 4 as a conclusion to the paper.

\section{Preliminaries}
\subsection{Two usual classes of growth functions and Bergman-Orlicz spaces}

Recall that a function $\Phi: [0,\infty) \to [0,\infty)$ is said to be a growth function if it is increasing, $\Phi(0)=0$, $\Phi(t)\to\infty$ as $t\to\infty$. We denote by $\mathcal{G}$ the set of all growth functions.

\begin{definition}\label{def:U}
$\Phi\in\mathcal{G}$ has upper type $q$ if there exists $C>0$ such that
\[
\Phi(st) \le C s^q \Phi(t), \qquad \forall s\ge 1,\ \forall t>0.
\]
Let $\U^q$ denote the set of growth functions of upper type $q$ such that $t\mapsto \Phi(t)/t$ is nondecreasing. Then
\[
\U := \bigcup_{q\ge 1} \U^q.
\]
\end{definition}

\begin{definition}\label{def:L}
$\Phi\in\mathcal{G}$ has lower type $p$ if there exists $C>0$ such that
\[
\Phi(st) \le C s^p \Phi(t), \qquad \forall 0<s<1,\ \forall t>0.
\]
Let ${\Lc}^p$ denote the set of growth functions of lower type $p$ such that $t\mapsto \Phi(t)/t$ is nonincreasing. Then
\[
\Lc := \bigcup_{p\le 1} {\Lc}^p.
\]
\end{definition}

Let us denote by $d\nu$ the Lebesgue measure on the unit ball $\mathbb B^n$ of $\mathbb C^n$. For $\alpha>-1$, we denote by $d\nu_{\alpha}$ the normalized Lebesgue measure $d\nu_{\alpha}(z)=c_{\alpha}(1-|z|^2)^{\alpha}d\nu(z)$, $c_\alpha$ being the normalization constant. In the upper half-plane of $\mathbb{C}$, we use the notation $dV_\alpha(x+iy)=y^\alpha dxdy$. 

Let $\Omega$ will be either the unit ball of $\mathbb{C}^n$ ($n\in\mathbb{N}$) or the upper half-plane of $\mathbb{C}$; we use the notation $d\Omega_\alpha$ for either $d\nu_{\alpha}(z)=c_{\alpha}(1-|z|^2)^{\alpha}d\nu(z)$ or $dV_\alpha(x+iy)$ . For $\Phi$ a growth function, the weighted Bergman-Orlicz space $A_\alpha^{\Phi}(\Omega)$ is the space of all holomorphic functions $f$ such that
$$\Vert f\Vert_{\Phi,\alpha}=||f||_{A_\alpha^{\Phi}}:=\int_{\Omega}\Phi(|f(z)|)d\Omega_{\alpha}(z)<\infty.$$
We define on $A_\alpha^{\Phi}(\Omega)$ the following (quasi)-norm
\begin{equation}\label{BergOrdef1}
\Vert f\Vert^{lux}_{\Phi,\alpha}=||f||^{lux}_{A_\alpha^{\Phi}}:=\inf\{\lambda>0: \int_{\Omega}\Phi\left(\frac{|f(z)|}{\lambda}\right)d\Omega_\alpha(z)\le 1\}.
\end{equation}
And associated space to the above, is the Hardy-Orlicz space $H^\Phi(\Omega)$ that we understand as the limit of $A_\alpha^\Phi(\Omega)$ when $\alpha\to -1^+$. For the specific definition, we refer to \cite{DieSehba2023}. 
\subsection{Log-convexity and log-concavity up to a constant}

\begin{definition}\label{def:logconv}
A growth function $\Phi$ is said to be log-convex up to a constant $C\ge 1$ (written $\Phi\in\Delta_{\log}^-$) if
\[
\Phi(u^\theta v^{1-\theta}) \le C^{\theta(1-\theta)} \Phi(u)^\theta \Phi(v)^{1-\theta}, \qquad \forall u,v>0,\ \theta\in[0,1].
\]
It is log-concave up to a constant $C\ge 1$ (written $\Phi\in\Delta_{\log}^+$) if
\[
\Phi(C^{\theta(1-\theta)} u^\theta v^{1-\theta}) \ge \Phi(u)^\theta \Phi(v)^{1-\theta}, \qquad \forall u,v>0,\ \theta\in[0,1].
\]
\end{definition}
We observe the following.
\begin{proposition}\label{prop:dualityconv}
If $\Phi$ is strictly increasing, then $\Phi \in \Delta_{\log}^-$ if and only if $\Phi^{-1} \in \Delta_{\log}^+$.
\end{proposition}
\begin{proof}
Let $\Phi\in\Delta_{\log}^-$. For any $u,v>0$, $\theta\in[0,1]$,
\[
\Phi(u^\theta v^{1-\theta}) \le C^{\theta(1-\theta)} \Phi(u)^\theta \Phi(v)^{1-\theta}.
\]
Apply $\Phi^{-1}$ to both sides:
\[
u^\theta v^{1-\theta} \le \Phi^{-1}\left( C^{\theta(1-\theta)} \Phi(u)^\theta \Phi(v)^{1-\theta} \right).
\]
Set $s=\Phi(u)$, $t=\Phi(v)$. Then $u=\Phi^{-1}(s)$, $v=\Phi^{-1}(t)$, so
\[
\Phi^{-1}(s)^\theta \Phi^{-1}(t)^{1-\theta} \le \Phi^{-1}\left( C^{\theta(1-\theta)} s^\theta t^{1-\theta} \right),
\]
which is exactly $\Phi^{-1}\in\Delta_{\log}^+$.
The converse is similar.
\end{proof}

Let us give some concrete examples of functions in $\Delta_{\log}^+$ and $\Delta_{\log}^-$.

\begin{example}
The power functions $\Phi(t) = t^p$ with $p>0$ belong to $\Delta_{\log}^+\cap \Delta_{\log}^-$.
\end{example}

\begin{example}
Let $\Phi(t) = e^t - 1$. Then $\Phi^{-1}(t) = \log(1+t)$. Thus, $\log\Phi^{-1}(t) = \log\log(1+t)$. Its second derivative is negative for $t>0$, so $\Phi^{-1}$ is log-concave. Hence,  $\Phi\in\Delta_{\log}^-$.
\end{example}

\begin{example}
Let $\Phi(t) = t^p \log^\alpha(1+t)$ with $p>1$, $\alpha>0$. It is easy to check that $\Phi$ is log concave, i.e. $\Phi\in \Delta_{\log}^+$.
\end{example}

\begin{example}
Let $\Phi(t) = \exp(\exp(t))-e$. Then $\Phi^{-1}(t) = \log\log(t+e)$. Then $\log\Phi^{-1}(t) = \log\log\log(t+e)$, which is concave (log of log of log). So $\Phi\in\Delta_{\log}^-$. 
\end{example}
\subsection{Interpolation of growth functions}
Let us start this subsection with the following observation.
\begin{proposition}\label{prop:interp}
Let $\Phi_0\in\U^{q_0}$ and $\Phi_1\in\U^{q_1}$ with $q_0\le q_1$. For
$\theta\in[0,1]$, define $\Phi_\theta$ by
\[
\log\Phi_\theta(t)=(1-\theta)\log\Phi_0(t)+\theta\log\Phi_1(t).
\]
Then $\Phi_\theta\in\U^{q_\theta}$ where $q_\theta=(1-\theta)q_0+\theta q_1$.
\end{proposition}

\begin{proof}
We need to verify both defining properties of $\U^{q_\theta}$.

For any $s>0$ and $t\ge 1$,
\begin{align*}
\Phi_\theta(st)
&= \Phi_0(st)^{1-\theta}\Phi_1(st)^\theta \\
&\le (C_0 t^{q_0}\Phi_0(s))^{1-\theta}(C_1 t^{q_1}\Phi_1(s))^\theta \\
&= C_0^{1-\theta}C_1^\theta \, t^{(1-\theta)q_0+\theta q_1} \, \Phi_0(s)^{1-\theta}\Phi_1(s)^\theta \\
&= C\, t^{q_\theta}\,\Phi_\theta(s),
\end{align*}
where $C=C_0^{1-\theta}C_1^\theta$. Thus $\Phi_\theta$ is of upper-type $q_\theta$.

Since $\log\Phi_0(t)/t$ and
$\log\Phi_1(t)/t$ are nondecreasing (as $\Phi_0,\Phi_1\in\U$), their convex
combination is also nondecreasing. More explicitly,
\[
\frac{d}{dt}\log\frac{\Phi_\theta(t)}{t}
= (1-\theta)\frac{d}{dt}\log\frac{\Phi_0(t)}{t}
+ \theta\frac{d}{dt}\log\frac{\Phi_1(t)}{t} \ge 0
\]
because each derivative is nonnegative.
Thus $\Phi_\theta(t)/t$ is nondecreasing.

Therefore $\Phi_\theta\in\U^{q_\theta}$.
\end{proof}
Let us also record the following.
\begin{proposition}\label{prop:interp-lower}
Let $\Phi_0\in\Lc^{p_0}$ and $\Phi_1\in\Lc^{p_1}$ with $p_0,p_1\le 1$. For
$\theta\in[0,1]$, define $\Phi_\theta$ by
\[
\log\Phi_\theta(t)=(1-\theta)\log\Phi_0(t)+\theta\log\Phi_1(t).
\]
Then $\Phi_\theta\in\Lc^{p_\theta}$ where $p_\theta=(1-\theta)p_0+\theta p_1$.
\end{proposition}

\begin{proof}
Let show that both defining properties of $\Lc^{p_\theta}$ are satisfied.

Since $\Phi_0\in\Lc^{p_0}$ and $\Phi_1\in\Lc^{p_1}$, there exist constants $C_0,C_1>0$ such that for all $0<s<1$ and $t>0$,
\[
\Phi_0(st)\le C_0 s^{p_0}\Phi_0(t),\qquad
\Phi_1(st)\le C_1 s^{p_1}\Phi_1(t).
\]
Then for any $0<s<1$ and $t>0$,
\begin{align*}
\Phi_\theta(st)
&= \Phi_0(st)^{1-\theta}\Phi_1(st)^{\theta} \\
&\le \bigl(C_0 s^{p_0}\Phi_0(t)\bigr)^{1-\theta}
   \bigl(C_1 s^{p_1}\Phi_1(t)\bigr)^{\theta} \\
&= C_0^{1-\theta}C_1^{\theta}\,
   s^{(1-\theta)p_0+\theta p_1}\,
   \Phi_0(t)^{1-\theta}\Phi_1(t)^{\theta} \\
&= C\, s^{p_\theta}\,\Phi_\theta(t),
\end{align*}
where $C=C_0^{1-\theta}C_1^{\theta}$ and $p_\theta=(1-\theta)p_0+\theta p_1$.
Thus $\Phi_\theta$ has lower type $p_\theta$.

For the second property, we note that
since $\Phi_0,\Phi_1\in\Lc$, the functions $t\mapsto\Phi_0(t)/t$ and $t\mapsto\Phi_1(t)/t$ are nonincreasing. Equivalently,
\[
\frac{d}{dt}\log\frac{\Phi_0(t)}{t}\le 0,\qquad
\frac{d}{dt}\log\frac{\Phi_1(t)}{t}\le 0.
\]
As
\[
\log\frac{\Phi_\theta(t)}{t}
= (1-\theta)\log\frac{\Phi_0(t)}{t}
+ \theta\log\frac{\Phi_1(t)}{t},
\]
we easily obtain
\[
\frac{d}{dt}\log\frac{\Phi_\theta(t)}{t}
= (1-\theta)\frac{d}{dt}\log\frac{\Phi_0(t)}{t}
+ \theta\frac{d}{dt}\log\frac{\Phi_1(t)}{t}
\le 0.
\]
Thus $\Phi_\theta(t)/t$ is nonincreasing.

Therefore $\Phi_\theta\in\Lc^{p_\theta}$.
\end{proof}
Let us recall the following definition.
\begin{definition}
For $\Phi\in\mathcal{C}^1(\R_+)$ a growth function, define
\[
a_\Phi:=\inf_{t>0}\frac{t\Phi'(t)}{\Phi(t)},\qquad
b_\Phi:=\sup_{t>0}\frac{t\Phi'(t)}{\Phi(t)}.
\]
These are called the Matuszewska--Orlicz indices; $a_\Phi$ is the lower indice, while $b_\Phi$ is the upper indice.
\end{definition}

We need the following to prove the next result.
\begin{lemma}\label{lem:inversemonocity}
Let $\Phi\in\Lc\cup\U$ and $\Psi\in\Lc\cup\U$ be such $a_\Psi\ge b_\Phi$. Then both that ${\Psi}/{\Phi}$ and ${\Phi^{-1}}/{\Psi^{-1}}$ are nondecreasing.
\end{lemma}
\begin{proof}
We have $\left(\frac{\Psi(t)}{\Phi(t)}\right)'=\frac{\Psi'(t)\Phi(t)-\Psi(t)\Phi'(t)}{\Phi^2(t)}$. From the definition of lower and upper indices, we obtain for $t>0$,
$$\Psi'(t)\Phi(t)-\Psi(t)\Phi'(t)=\frac{\Psi(t)\Phi(t)}{t}\left(\frac{t\Psi'(t)}{\Psi(t)}-\frac{t\Phi'(t)}{\Phi(t)}\right)\ge \frac{\Psi(t)\Phi(t)}{t}(a_{\Psi}-b_{\Phi})\ge 0.$$
That is ${\Psi}/{\Phi}$ is nondecreasing. The proof for ${\Phi^{-1}}/{\Psi^{-1}}$ follows the same way, using the relations $a_{\Phi^{-1}}=\frac{1}{b_\Phi}$ and $b_{\Phi^{-1}}=\frac{1}{a_\Phi}$ (see \cite{DieSehba2023}).  
\end{proof}
The following result ensures that the natural order of growth functions (expressed via the monotonicity of the ratio $\Psi/\Phi$) is preserved under the log-convex interpolation of inverses. 

\begin{proposition}\label{prop:monotonicity-preservation}
Let $\Phi_0,\Phi_1,\Psi_0,\Psi_1$ be growth functions such that $a_{\Psi_i}\ge b_{\Phi_i}$, $i=0,1$.
Define the log-convex interpolations
\[
\Phi_\theta^{-1}(t) = \Phi_0^{-1}(t)^{1-\theta} \Phi_1^{-1}(t)^{\theta}, \quad
\Psi_\theta^{-1}(t) = \Psi_0^{-1}(t)^{1-\theta} \Psi_1^{-1}(t)^{\theta}.
\]
Then the ratios $R_\theta(t) = \Psi_\theta(t)/\Phi_\theta(t)$ and $\tilde{R}_\theta(t) = \Psi_\theta^{-1}(t)/\Phi_\theta^{-1}(t)$ are  nondecreasing on $\R_+$ for every $\theta \in [0,1]$.
\end{proposition}

\begin{proof}
By the previous lemma, $\frac{\Psi}{\Phi}$ and $\frac{\Phi^{-1}}{\Psi^{-1}}$ are both nondecreasing. Let us show that ratio $\tilde{R}_\theta(t) = \Psi_\theta^{-1}(t)/\Phi_\theta^{-1}(t)$ is nondecreasing on $\R_+$ for every $\theta \in [0,1]$.
Observe that
\[
\log \tilde{R}_\theta(t) = (1-\theta)\log \tilde{R}_0(t) + \theta \log \tilde{R}_1(t),
\]
where $\tilde{R}_i=\frac{\Phi_i^{-1}}{\Psi_i^{-1}}$, $i=0,1$. Since $\tilde{R}_0$ and $\tilde{R}_1$ are nondecreasing and positive, for $s \le t$ we have
\[
\frac{\tilde{R}_\theta(t)}{\tilde{R}_\theta(s)} = \left(\frac{\tilde{R}_0(t)}{\tilde{R}_0(s)}\right)^{1-\theta}
\left(\frac{\tilde{R}_1(t)}{\tilde{R}_1(s)}\right)^{\theta} \ge 1^{1-\theta} \cdot 1^{\theta} = 1.
\]
Thus $\tilde{R}_\theta(t) \ge \tilde{R}_\theta(s)$, so $\tilde{R}_\theta$ is nondecreasing. That $R$ is nondecreasing follows the same way. The proof is complete.
\end{proof}

\subsection{The three‑lines lemma}
We recall the following Hadamard's three‑lines lemma. We will use it to give an alternative proof of one of the convexity results. For a proof, see e.g. \cite[Chapter I, Theorem 5.1]{Garnett}.
\begin{lemma}\label{lem:threelines}
Let $F(z)$ be analytic on the open strip $0 < \operatorname{Re} z < 1$ and continuous on its closure, and suppose $|F(z)| \le M_0$ on the line $\operatorname{Re} z = 0$ and $|F(z)| \le M_1$ on the line $\operatorname{Re} z = 1$. Then for any $z$ with $\operatorname{Re} z = \theta$ ($0<\theta<1$),
\[
|F(z)| \le M_0^{1-\theta} M_1^{\theta}.
\]
\end{lemma}

\section{Main results}
We state and prove our main results in this section.

\subsection{Convexity of $\mathcal{E}$} 

\begin{theorem}\label{thm:convexity}
Let $\Phi_1$ be a growth function such that $\Phi_1^{-1}\in \Delta_{\log}^-$ with constant $C_1\geq 1$, and let $\Phi_2$ be a growth function such that $\Phi_2^{-1}\in \Delta_{\log}^+$ with constant
$C_2 \ge 1$. Then
the set
\[
\mathcal{E} = \left\{ (\alpha,\beta) \in \R^2 \;\middle|\;
\alpha, \beta \ge -1, \quad
\exists\, C,K > 0 \;\text{ s.t. }\;
\Phi_1^{-1}(t^{2+\alpha}) \le C\,\Phi_2^{-1}(Kt^{2+\beta})
\;\; \forall\, t > 0
\right\}
\]
is convex.
\end{theorem}

\begin{proof}
Let $(\alpha_0, \beta_0), (\alpha_1, \beta_1) \in \mathcal{E}$. By definition,
there exist constants $M_0, M_1 > 0$ and $K_0, K_1 > 0$ such that for all $t > 0$,
\begin{align}
\Phi_1^{-1}\!\left( t^{2+\alpha_0} \right) &\le M_0\,\Phi_2^{-1}\!\left( K_0 t^{2+\beta_0} \right), \label{eq:e0}\\
\Phi_1^{-1}\!\left( t^{2+\alpha_1} \right) &\le M_1\,\Phi_2^{-1}\!\left( K_1 t^{2+\beta_1} \right). \label{eq:e1}
\end{align}
Fix $\theta \in [0,1]$ and set
\[
\alpha_\theta = (1-\theta)\alpha_0 + \theta\alpha_1, \qquad
\beta_\theta = (1-\theta)\beta_0 + \theta\beta_1.
\]
Let us show that $(\alpha_\theta, \beta_\theta) \in \mathcal{E}$.

For any $t > 0$, write
\[
t^{2+\alpha_\theta}
= \left(t^{2+\alpha_0}\right)^{1-\theta}
\cdot \left(t^{2+\alpha_1}\right)^{\theta}.
\]
Since $\Phi_1^{-1}\in\Delta_{\log}^-$ with constant $C_1\ge 1$, it holds that for all $u,v>0$,
\[
\Phi_1^{-1}\!\left(u^{1-\theta}v^{\theta}\right) \le C_1^{\theta(1-\theta)}\,
\Phi_1^{-1}(u)^{1-\theta}\,\Phi_1^{-1}(v)^{\theta}.
\]
Applying this with $u = t^{2+\alpha_0}$ and $v = t^{2+\alpha_1}$ gives
\begin{equation}\label{eq:step1}
\Phi_1^{-1}(t^{2+\alpha_\theta})
\le C_1^{\theta(1-\theta)}\,
\Phi_1^{-1}(t^{2+\alpha_0})^{1-\theta}\,
\Phi_1^{-1}(t^{2+\alpha_1})^{\theta}.
\end{equation}

Raising \eqref{eq:e0} to the power $(1-\theta)$ and \eqref{eq:e1}
to the power $\theta$, gives 
\begin{align}
\Phi_1^{-1}(t^{2+\alpha_0})^{1-\theta}
&\le M_0^{1-\theta}\,\Phi_2^{-1}(K_0 t^{2+\beta_0})^{1-\theta},\label{eq:pow0}\\
\Phi_1^{-1}(t^{2+\alpha_1})^{\theta}
&\le M_1^{\theta}\,\Phi_2^{-1}(K_1 t^{2+\beta_1})^{\theta}.\label{eq:pow1}
\end{align}
Then multiplying \eqref{eq:pow0} and \eqref{eq:pow1}, we get
\begin{equation}\label{eq:step2}
\Phi_1^{-1}(t^{2+\alpha_0})^{1-\theta}\,
\Phi_1^{-1}(t^{2+\alpha_1})^{\theta}
\le M_0^{1-\theta} M_1^{\theta}\,
\Phi_2^{-1}(K_0 t^{2+\beta_0})^{1-\theta}\,
\Phi_2^{-1}(K_1 t^{2+\beta_1})^{\theta}.
\end{equation}

Since $\Phi_2^{-1}\in\Delta_{\log}^+$ with constant $C_2\ge 1$, it holds that for all $u,v>0$,
\[
\Phi_2^{-1}(u)^{1-\theta}\Phi_2^{-1}(v)^{\theta}
\le 
\Phi_2^{-1}\!\left(C_2^{\theta(1-\theta)}\,u^{1-\theta}v^{\theta}\right).
\]
We apply this with $u = K_0 t^{2+\beta_0}$ and $v = K_1 t^{2+\beta_1}$. For this, we observe that
\[
u^{1-\theta}v^{\theta}
= (K_0 t^{2+\beta_0})^{1-\theta} (K_1 t^{2+\beta_1})^{\theta}
= K_0^{1-\theta}K_1^{\theta} \,
t^{(1-\theta)(2+\beta_0)+\theta(2+\beta_1)},
\]
and $(1-\theta)(2+\beta_0)+\theta(2+\beta_1) = 2+\beta_\theta$. Therefore
\[
u^{1-\theta}v^{\theta} = L_\theta \, t^{2+\beta_\theta},
\]
where we set $L_\theta := K_0^{1-\theta}K_1^{\theta}$. Consequently,
\begin{equation}\label{eq:step3}
\Phi_2^{-1}(K_0 t^{2+\beta_0})^{1-\theta}\,
\Phi_2^{-1}(K_1 t^{2+\beta_1})^{\theta}
\le 
\Phi_2^{-1}\!\left(C_2^{\theta(1-\theta)}\,L_\theta t^{2+\beta_\theta}\right).
\end{equation}

Substituting \eqref{eq:step3} into \eqref{eq:step2} yields
\[
\Phi_1^{-1}(t^{2+\alpha_0})^{1-\theta}\,
\Phi_1^{-1}(t^{2+\alpha_1})^{\theta}
\le M_0^{1-\theta} M_1^{\theta} C_2^{\theta(1-\theta)}\,
\Phi_2^{-1}\!\left(C_2^{\theta(1-\theta)}\,L_\theta t^{2+\beta_\theta}\right).
\]
Finally, inserting this into \eqref{eq:step1} gives
\[
\Phi_1^{-1}(t^{2+\alpha_\theta})
\le C_1^{\theta(1-\theta)} M_0^{1-\theta} M_1^{\theta} C_2^{\theta(1-\theta)}\,
\Phi_2^{-1}\!\left(C_2^{\theta(1-\theta)}\,L_\theta t^{2+\beta_\theta}\right)
= C_\theta \, \Phi_2^{-1}\!\left(K_\theta t^{2+\beta_\theta}\right),
\]
where $C_\theta := M_0^{1-\theta} M_1^{\theta} C_1^{\theta(1-\theta)}$ and $K_\theta=C_2^{\theta(1-\theta)}\,L_\theta$. This shows $(\alpha_\theta,\beta_\theta) \in \mathcal{E}$. The proof is complete.
\end{proof}

We derive the following fact.
\begin{corollary}
The set $\mathcal{E}$ is an epigraph. More precisely, there exists a convex, nondecreasing function $\beta^*: [-1,\infty) \to [-1,\infty]$ such that
\[
\mathcal{E} = \{ (\alpha,\beta) : \alpha \ge -1,\ \beta \ge \beta^*(\alpha) \}.
\]
Moreover, $\beta^*$ is given by
\[
\beta^*(\alpha) = \inf\left\{ \beta \ge -1 \;:\; \exists C,K>0,\ 
\Phi_1^{-1}(t^{2+\alpha}) \le C\,\Phi_2^{-1}(Kt^{2+\beta}) \ \forall t>0 \right\}.
\]
\end{corollary}

\begin{proof}
We first note that if $(\alpha,\beta_0)\in\mathcal{E}$ and $\beta_1 \ge \beta_0$, then $(\alpha,\beta_1)\in\mathcal{E}$. Indeed, for $t>0$, $t^{2+\beta_1} \ge t^{2+\beta_0}$ (since $2+\beta_1 \ge 2+\beta_0$). Because $\Phi_2^{-1}$ is increasing, we have $\Phi_2^{-1}(Kt^{2+\beta_1}) \ge \Phi_2^{-1}(Kt^{2+\beta_0})$. Hence the inequality $\Phi_1^{-1}(t^{2+\alpha}) \le C\Phi_2^{-1}(Kt^{2+\beta_0})$ implies the same with $\beta_1$ in place of $\beta_0$ (using the same $C$). Therefore $\mathcal{E}$ is an upward-closed set in the $\beta$ direction. This allows us to define the boundary function $\beta^*(\alpha)$ as the infimum over all $\beta$ such that $(\alpha,\beta)\in\mathcal{E}$.

\medskip
The convexity of $\mathcal{E}$ (Theorem~\ref{thm:convexity}) implies that $\beta^*$ is convex. Indeed, for any $\alpha_0,\alpha_1$ and $\theta\in[0,1]$, take any $\beta_0\ge\beta^*(\alpha_0)$ and $\beta_1\ge\beta^*(\alpha_1)$. Then $(\alpha_0,\beta_0),(\alpha_1,\beta_1)\in\mathcal{E}$, so by convexity $(\alpha_\theta,\beta_\theta)\in\mathcal{E}$ with $\beta_\theta = (1-\theta)\beta_0+\theta\beta_1$. Taking infimum over $\beta_0,\beta_1$ yields $\beta^*(\alpha_\theta) \le (1-\theta)\beta^*(\alpha_0)+\theta\beta^*(\alpha_1)$, i.e., $\beta^*$ is convex.

\medskip
Let us check that $\beta^*$ is nondecreasing. If $\alpha_1 \ge \alpha_0$ and $(\alpha_0,\beta)\in\mathcal{E}$, then for $t>0$, $t^{2+\alpha_1} \ge t^{2+\alpha_0}$, so $\Phi_1^{-1}(t^{2+\alpha_1}) \ge \Phi_1^{-1}(t^{2+\alpha_0})$. This implies for any $\beta_1$ such that $\Phi_1^{-1}(t^{2+\alpha_1})\leq C\Phi_2^{-1}(Kt^{2+\beta_1})$, we also have $\Phi_1^{-1}(t^{2+\alpha_0})\leq C\Phi_2^{-1}(Kt^{2+\beta_1})$. Hence $\beta^*(\alpha_0)\leq \beta_1$. This implies $\beta^*(\alpha_1) \ge \beta^*(\alpha_0)$. Thus $\beta^*$ is nondecreasing. The proof is complete.
\end{proof}

\subsection{Log-Convexity of $\mathcal{F}$ and interpolation of Bergman-Orlicz space embeddings}

We now consider the dual problem: fix the weights $\alpha,\beta \ge -1$ and
consider the set of growth function pairs.

\begin{theorem}[Inverse log-convexity of $\mathcal{F}$]\label{thm:inverse-convexity-F}
Let $\alpha,\beta \ge -1$ be fixed. Define
\[
\mathcal{F} = \left\{ (\Phi_1,\Phi_2) \in \mathcal{G}^2 \;\middle|\;
\exists\, C > 0 \text{ such that }
\Phi_1^{-1}\!\left( t^{2+\alpha} \right) \le C\,
\Phi_2^{-1}\!\left( t^{2+\beta} \right) \;\; \forall\, t>0
\right\}.
\]
Then $\mathcal{F}$ is inverse log-convex. That is, if $(\Phi_0,\Psi_0), (\Phi_1,\Psi_1) \in \mathcal{F}$,
then for any $\theta \in (0,1)$ the pair $(\Phi_\theta,\Psi_\theta)$ defined by
\[
\Phi_\theta^{-1}(t) = \Phi_0^{-1}(t)^{1-\theta} \, \Phi_1^{-1}(t)^{\theta},
\qquad
\Psi_\theta^{-1}(t) = \Psi_0^{-1}(t)^{1-\theta} \, \Psi_1^{-1}(t)^{\theta}
\]
also belongs to $\mathcal{F}$.
\end{theorem}

\begin{proof}[First proof: through the three-lines lemma]
Let $t > 0$ be arbitrary. For each fixed $t$, define the functions
\[
u_0(t) = \log \Phi_0^{-1}(t^{2+\alpha}) - \log \Psi_0^{-1}(t^{2+\beta}),
\]
\[
u_1(t) = \log \Phi_1^{-1}(t^{2+\alpha}) - \log \Psi_1^{-1}(t^{2+\beta}).
\]
The hypothesis $(\Phi_0,\Psi_0) \in \mathcal{F}$ implies that there exists a constant
$C_0 > 0$ such that for all $t > 0$,
\begin{equation}\label{eq:firstproof1}
u_0(t) \le \log C_0. 
\end{equation}
Similarly, there exists $C_1 > 0$ such that for all $t > 0$,
\begin{equation}\label{eq:firstproof2}
u_1(t) \le \log C_1. 
\end{equation}

For $z \in \mathbb{C}$ with $0 \le \operatorname{Re} z \le 1$, define
\[
\Phi_z^{-1}(t) = \Phi_0^{-1}(t)^{1-z} \, \Phi_1^{-1}(t)^{z},
\qquad
\Psi_z^{-1}(t) = \Psi_0^{-1}(t)^{1-z} \, \Psi_1^{-1}(t)^{z}.
\]
For each fixed $t > 0$, the functions
\[
\log \Phi_z^{-1}(t) = (1-z)\log \Phi_0^{-1}(t) + z \log \Phi_1^{-1}(t),
\]
\[
\log \Psi_z^{-1}(t) = (1-z)\log \Psi_0^{-1}(t) + z \log \Psi_1^{-1}(t)
\]
are affine in $z$. Consequently, the function
\[
G_t(z) = \exp\Big( \log \Phi_z^{-1}(t^{2+\alpha}) - \log \Psi_z^{-1}(t^{2+\beta}) \Big)
= \Phi_z^{-1}(t^{2+\alpha}) \cdot \Psi_z^{-1}(t^{2+\beta})^{-1}
\]
is an exponential of an affine function in $z$, hence analytic in $z$.

On the line $z = it$, we have
\[
\log \Phi_{it}^{-1}(t^{2+\alpha}) = (1-it)\log \Phi_0^{-1}(t^{2+\alpha}) + it \log \Phi_1^{-1}(t^{2+\alpha}).
\]
The real part is $\log \Phi_0^{-1}(t^{2+\alpha})$. Similarly,
the real part of $\log \Psi_{it}^{-1}(t^{2+\beta})$ is $\log \Psi_0^{-1}(t^{2+\beta})$.
Therefore,
\[
|G_t(it)| = \exp\Big( \log \Phi_0^{-1}(t^{2+\alpha}) - \log \Psi_0^{-1}(t^{2+\beta}) \Big)
= \frac{\Phi_0^{-1}(t^{2+\alpha})}{\Psi_0^{-1}(t^{2+\beta})} \le C_0,
\]
by (\ref{eq:firstproof1}). On the line $z = 1+it$,
\[
|G_t(1+it)| = \frac{\Phi_1^{-1}(t^{2+\alpha})}{\Psi_1^{-1}(t^{2+\beta})} \le C_1,
\]
by (\ref{eq:firstproof2}).

For each fixed $t > 0$, $G_t(z)$ is analytic on the strip  $0\leq \Re z\leq 1$ and uniformly bounded
(the bound depends on $t$ but is finite pointwise). The three-lines lemma gives,
for $z = \theta$,
\[
|G_t(\theta)| \le C_0^{1-\theta} C_1^{\theta}.
\]

But $G_t(\theta)$ is precisely
\[
G_t(\theta) = \Phi_\theta^{-1}(t^{2+\alpha}) \cdot \Psi_\theta^{-1}(t^{2+\beta})^{-1},
\]
where $\Phi_\theta^{-1}$ and $\Psi_\theta^{-1}$ are defined as in the theorem statement.
Thus, for all $t > 0$,
\[
\Phi_\theta^{-1}(t^{2+\alpha}) \le C_\theta \, \Psi_\theta^{-1}(t^{2+\beta}),
\]
with $C_\theta = C_0^{1-\theta} C_1^{\theta}$.

This is exactly the condition that $(\Phi_\theta,\Psi_\theta) \in \mathcal{F}$.
Therefore, $\mathcal{F}$ is inverse log-convex.
\end{proof}

\begin{proof}[Second proof: an algebraic approach]
This proof is straightforward and uses only elementary algebra.

Since $(\Phi_0,\Psi_0) \in \mathcal{F}$, there exists $C_0 > 0$ such that for all $t > 0$,
\begin{equation}\label{eq:secondproof1}
\Phi_0^{-1}\!\left( t^{2+\alpha} \right) \le C_0 \, \Psi_0^{-1}\!\left( t^{2+\beta} \right).
\end{equation}

Since $(\Phi_1,\Psi_1) \in \mathcal{F}$, there exists $C_1 > 0$ such that for all $t > 0$,
\begin{equation}\label{eq:secondproof2}
\Phi_1^{-1}\!\left( t^{2+\alpha} \right) \le C_1 \, \Psi_1^{-1}\!\left( t^{2+\beta} \right). 
\end{equation}

Now fix $\theta \in (0,1)$. Raise inequality (\ref{eq:secondproof1}) to the power $1-\theta$ and inequality (\ref{eq:secondproof2}) to the power $\theta$ (all quantities are positive, so this preserves the inequalities):
\[
\Phi_0^{-1}\!\left( t^{2+\alpha} \right)^{1-\theta} \le C_0^{1-\theta} \, \Psi_0^{-1}\!\left( t^{2+\beta} \right)^{1-\theta},
\]
\[
\Phi_1^{-1}\!\left( t^{2+\alpha} \right)^{\theta} \le C_1^{\theta} \, \Psi_1^{-1}\!\left( t^{2+\beta} \right)^{\theta}.
\]

Multiply the two inequalities, we obtain
\[
\Phi_0^{-1}\!\left( t^{2+\alpha} \right)^{1-\theta} \Phi_1^{-1}\!\left( t^{2+\alpha} \right)^{\theta}
\le C_0^{1-\theta} C_1^{\theta} \,
\Psi_0^{-1}\!\left( t^{2+\beta} \right)^{1-\theta} \Psi_1^{-1}\!\left( t^{2+\beta} \right)^{\theta}.
\]

But by definition of the interpolated inverses,
\[
\Phi_\theta^{-1}\!\left( t^{2+\alpha} \right) = \Phi_0^{-1}\!\left( t^{2+\alpha} \right)^{1-\theta} \Phi_1^{-1}\!\left( t^{2+\alpha} \right)^{\theta},
\]
\[
\Psi_\theta^{-1}\!\left( t^{2+\beta} \right) = \Psi_0^{-1}\!\left( t^{2+\beta} \right)^{1-\theta} \Psi_1^{-1}\!\left( t^{2+\beta} \right)^{\theta}.
\]

Therefore,
\[
\Phi_\theta^{-1}\!\left( t^{2+\alpha} \right) \le C_\theta \, \Psi_\theta^{-1}\!\left( t^{2+\beta} \right),
\]
where $C_\theta = C_0^{1-\theta} C_1^{\theta}$.

Thus $(\Phi_\theta,\Psi_\theta) \in \mathcal{F}$. The proof is complete.
\end{proof}

We observe the following about the interpolation constant in the above result.
\begin{corollary}
For fixed $(\alpha,\beta)$ and fixed endpoint pairs $(\Phi_0,\Psi_0), (\Phi_1,\Psi_1) \in \mathcal{F}$,
the minimal constant
\[
C_{\min}(\Phi,\Psi) = \sup_{t>0} \frac{\Phi^{-1}(t^{2+\alpha})}{\Psi^{-1}(t^{2+\beta})}
\]
satisfies
\[
C_{\min}(\Phi_\theta,\Psi_\theta) \le C_{\min}(\Phi_0,\Psi_0)^{1-\theta} \,
C_{\min}(\Phi_1,\Psi_1)^{\theta}.
\]
Thus $\log C_{\min}$ is a convex function along log-convex interpolations.
\end{corollary}

\begin{proof}
From the proof of Theorem~\ref{thm:inverse-convexity-F}, for all $t>0$,
\[
\frac{\Phi_\theta^{-1}(t^{2+\alpha})}{\Psi_\theta^{-1}(t^{2+\beta})}
\le \left( \sup_{s>0} \frac{\Phi_0^{-1}(s^{2+\alpha})}{\Psi_0^{-1}(s^{2+\beta})} \right)^{1-\theta}
\left( \sup_{s>0} \frac{\Phi_1^{-1}(s^{2+\alpha})}{\Psi_1^{-1}(s^{2+\beta})} \right)^{\theta}.
\]
Taking the supremum over $t>0$ on the left-hand side yields the desired inequality.
\end{proof}

\begin{example}
Let $\Phi_0(t)=t^{p_0}$, $\Phi_1(t)=t^{p_1}$, $\Psi_0(t)=t^{q_0}$, $\Psi_1(t)=t^{q_1}$.
Then $\Phi_\theta^{-1}(t)=t^{1/p_\theta}$ with $1/p_\theta = (1-\theta)/p_0 + \theta/p_1$,
and similarly $\Psi_\theta^{-1}(t)=t^{1/q_\theta}$.
The condition $(\Phi_i,\Psi_i) \in \mathcal{F}$ becomes
\[
t^{(2+\alpha)/p_i} \le C_i t^{(2+\beta)/q_i},\quad t\ge 1 \quad \Longleftrightarrow \quad
\frac{2+\alpha}{p_i} \le \frac{2+\beta}{q_i}.
\]
The interpolated condition follows immediately:
\[
\frac{2+\alpha}{p_\theta} = (1-\theta)\frac{2+\alpha}{p_0} + \theta\frac{2+\alpha}{p_1}
\le (1-\theta)\frac{2+\beta}{q_0} + \theta\frac{2+\beta}{q_1}
= \frac{2+\beta}{q_\theta}.
\]
The algebraic proof reproduces this with $C_\theta = C_0^{1-\theta}C_1^{\theta}=1^{1-\theta}1^\theta=1$.
\end{example}

\medskip
Theorem~\ref{thm:inverse-convexity-F} has a direct and important application to the interpolation of embeddings between Bergman-Orlicz spaces. Recall from the work of Dje and Sehba \cite{DieSehba2021,DieSehba2023,Sehba} that for $\Phi, \Psi \in \U\cap \Lc$, $\alpha\ge -1$,and $\beta > -1$, under the additional hypothesis that the ratio $\Psi(t)/\Phi(t)$ is nondecreasing (which guarantees the continuous embedding of the underlying Orlicz spaces), the embedding
\[
A^{\Phi}_{\alpha}(\Omega) \hookrightarrow A^{\Psi}_{\beta}(\Omega)
\]
holds if and only if there exists a constant $C>0$ such that
\[
\Phi^{-1}(t^{2+\alpha}) \le C \, \Psi^{-1}(t^{2+\beta}) \qquad \forall t\ge 1.
\]
Thus the condition defining $\mathcal{F}$ together with the monotonicity of $\Psi/\Phi$ is exactly the embedding condition. 
We then obtain the following interpolation of Bergman-Orlicz spaces result.
\begin{theorem}\label{thm:interpolation}
Let $\alpha\ge -1$, and $\beta > -1$ be fixed. Suppose that for two pairs of growth functions $(\Phi_0,\Psi_0)\in \Lc\cup\U$ and $(\Phi_1,\Psi_1)\in\Lc\cup\U$ we have $a_{\Psi_i}\ge b_{\Phi_i}$, $i-0,1$.
Then for any $\theta \in (0,1)$, the interpolated growth functions defined by
\[
\Phi_\theta^{-1}(t) = \Phi_0^{-1}(t)^{1-\theta} \Phi_1^{-1}(t)^{\theta}, \quad
\Psi_\theta^{-1}(t) = \Psi_0^{-1}(t)^{1-\theta} \Psi_1^{-1}(t)^{\theta}
\]
satisfy
\begin{itemize}
\item The ratio $\Psi_\theta(t)/\Phi_\theta(t)$ is nondecreasing,
\item The embedding $A^{\Phi_\theta}_{\alpha}(\Omega) \hookrightarrow A^{\Psi_\theta}_{\beta}(\Omega)$ holds.
\end{itemize}
\end{theorem}

\begin{proof}
The equivalence of the embedding to the inverse inequality is given in \cite{DieSehba2021,DieSehba2023,Sehba} under the monotonicity assumption. Theorem~\ref{thm:inverse-convexity-F} provides the inverse inequality for the interpolated pair. Proposition~\ref{prop:monotonicity-preservation} ensures that the required monotonicity of the ratio is preserved. Hence the interpolated pair satisfies all hypotheses of \cite[Corollary 2.6]{DieSehba2023}, and the embedding follows. The proof is complete.
\end{proof}

\section{Conclusion}
The two main results of this paper show that the convexity of $\mathcal{E}$ in the weight parameters $(\alpha,\beta)$ relies on the log-convexity/log-concavity properties of the \emph{inverses} of the growth functions, while the log-convexity of $\mathcal{F}$ in the growth functions $(\Phi_1,\Phi_2)$ is a direct consequence of the algebraic structure of the defining inequality.

\medskip
The log-convex interpolation of inverses used in Theorem~\ref{thm:inverse-convexity-F} is exactly the complex interpolation method applied to the scale of Orlicz spaces. If one defines the Orlicz space $L^{\Phi_\theta}$ via the Young function $\Phi_\theta$ constructed from $\Phi_0$ and $\Phi_1$ by this formula, then the complex interpolation space $[L^{\Phi_0}, L^{\Phi_1}]_\theta$ is equal to $L^{\Phi_\theta}$ (up to equivalence of norms), provided the growth functions satisfy appropriate conditions (\cite{BerghLofstrom1976,delCampo2014,GustavssonPeetre1977}). To then obtain $A^{\Phi_\theta}$ as complex interpolate of two Bergman-Orlicz spaces, one needs the growth functions to also satisfy  the $\nabla_2$-condition so that a Bergman projection is continuous on the corresponding Orlicz space  (see \cite{DieSehba2021,DieSehba2023} for the definitions). Thus Theorem~\ref{thm:interpolation} can also be obtained by applying complex interpolation to the embedding operator provided one restricts the class $\U\cap\nabla_2$ (by analogy with the power functions case \cite{ZhuZhao} and the requirement of the boundedness of the Bergman projection in \cite{DieSehba2023}). However, our direct algebraic approach avoids any functional analysis and works directly at the level of growth functions, not only making the proof elementary and self-contained, but allowing a wider class of growth functions. For example, our result gives an answer even for growth functions in $\Lc$ while complex interpolation is usually done only for example in case of power functions $\Phi(t)=t^p$, only for $1\le p<\infty$ (see \cite{ZhuZhao}).

\end{document}